\documentclass[12pt]{amsart}
\usepackage{graphicx}
\usepackage{amsmath}
\usepackage{subfigure}
\usepackage{amssymb} 
\usepackage{pinlabel}
\usepackage{mathtools}
\usepackage{txfonts} 
\usepackage{hyperref}
\hypersetup{linktocpage}

\usepackage{float}
\usepackage{tikz}
\usepackage{imakeidx}
\makeindex

  \usetikzlibrary{hobby}
  \usetikzlibrary{patterns}
    \input xy
  \xyoption{all}

\usepackage{fancyhdr}
\newtheorem{thm}{Theorem}[section]  
\newtheorem{prop}[thm]{Proposition}
\newtheorem{lemma}[thm]{Lemma} 
\newtheorem{cor}[thm]{Corollary}

\theoremstyle{definition}

\newtheorem{defin}[thm]{Definition}

\newcommand{\aaa}{\mbox{$\alpha$}}

\newcommand{\ddd}{\mbox{$\delta$}} 

\newcommand{\eee}{\mbox{$\epsilon$}}

\newcommand{\tM}{\tilde{M}} 
\newcommand{\tD}{\tilde{D}} 
\newcommand{\tE}{\tilde{E}} 
\newcommand{\tx}{\tilde{x}} 
\newcommand{\tb}{\tilde{b}}

\newcommand{\tX}{\tilde{X}}
\newcommand{\tL}{\tilde{L}}

\newcommand{\Rrr}{\mathbb R}
\newcommand{\Zzz}{\mathbb Z}

\begin{document}  

\title{The kernel of the Goldberg homomorphism is not finitely generated}   

\author{Martin Scharlemann}
\address{\hskip-\parindent
        Martin Scharlemann\\
        Mathematics Department\\
        University of California\\
        Santa Barbara, CA 93106-3080 USA}
\email{mgscharl@math.ucsb.edu}

\date{\today}

\begin{abstract}  Let $M$ be a closed surface not the sphere or projective plane.  Goldberg \cite{Gol} defined a natural homomorphism from the $n$-stranded pure braid group of $M$ to the n-fold product of $\pi_1(M)$ and showed that the kernel of the homomorphism  is finitely normally generated.  Here we show that the kernel is {\em not} finitely generated.  The proof is an elementary application of covering space theory and the  geometry of the euclidean or hyperbolic plane.  \end{abstract}

\maketitle

\section{Introduction} \label{sect:intro}

Let  $X = \{x_1, x_2, ..., x_n\}, n \geq 2$ be a set of $n$ points in a closed surface $M$ that is not the sphere or projective plane.  Let $F_nM$ be the space of all embeddings of $X$ in $M$, with $\{x^0: X \to M\} \in F_nM$ the inclusion.  Then $\pi_1(F_nM, x^0)$ is commonly called the pure braid group $P_n(M)$ of $M$ on $n$ strands.  Goldberg \cite{Gol} defines a natural surjection 
 \[g_*: P_n(M) \to \Pi_{i=1}^n \pi_1(M, x_i)\]
 and shows that 
 the kernel $K$ of $g_*$ is the normal closure of $P_n(D)$,  $D \subset M$ any disk containing $X$ in its interior. (See also \cite[Theorem 1.7]{Bi}.) The map $g_*$ can be thought of as induced by the inclusion $(F_nM, x^0) \to  (M^n, x^0)$.
  
 \begin{cor}[Goldberg] \label{cor:goldberg}  $K$ is finitely {\textbf normally} generated.
 \end{cor} 
 
 In contrast, we will here show
 
 \begin{thm} \label{thm:main} $K$ is not finitely generated.
 \end{thm}
 
Proving that a group is not finitely generated can be difficult.  See \cite{CF} for an example reminiscent of Theorem \ref{thm:main}, and  \cite{BF} for an example of a sophisticated proof gone wrong.  The proof of Theorem \ref{thm:main} given here is elementary, requiring only covering space theory and simple euclidean or hyperbolic geometry in the plane.  Indeed, the proof can be viewed as a re-imagining of Goldberg's 1973 proof in this journal, one in which it is the geometry rather than just the topology of the universal cover $\tM$ of $M$ that is invoked.  In view of this connection, we adopt much of the notation of Goldberg's proof, at least where it does not conflict with current sensibilities, and restrict the theorem, as Goldberg does, to closed surfaces.  (Requiring only that $M$ have infinite fundamental group appears to suffice.)  But we also note that Theorem \ref{thm:main} follows immediately from the more sophisticated structure theorem for $K$ derived in \cite[Proposition 2.4]{GP}  (see also \cite{Xi}) or, as the referee points out, can be derived from the general theory of covering spaces and subspace arrangements.  
\smallbreak

{\textbf Remarks:} The author was drawn to the question by efforts to prove that the Goeritz group  \cite{Sc} for Heegaard genus $\geq 4$ is not finitely generated.  Such a proof would allow us to conclude that  the Powell Conjecture is true only for genus $\leq 3$ \cite{FS}. 

\section{Winding number in $P_2(\Rrr^2)$}

The pure 2-braid group $P_2(\Rrr^2)$ of the plane is isomorphic to the integers:  the isomorphism $\omega: P_2(\Rrr^2) \to \Zzz$ assigns to a braid class the winding number of any braid in that braid class.  When $\Rrr^2$ is endowed with either the euclidean or hyperbolic metric there is a useful geometric definition of the winding number.  Let $b: I \to F_2\Rrr^2$ be a pure 2-braid in either the euclidean or hyperbolic plane.  For each $t \in I$ consider the geodesic ray $r(t)$ from $b(t)(x_1)$ to $b(t)(x_2)$.  In the Euclidean case, $r(t)$ defines an angle $w(t) \in S^1$; in the hyperbolic case the ray $r(t)$ determines a point $w(t) \in S^1$, viewed as the circle at infinity in the Poincar\'e disk model of the hyperbolic plane.  Thus in either case the pure braid $b(t)$ determines a map $w: I \to S^1$, with $w(0) = w(1)$.  Since $w(0) = w(1)$, $w$ determines a map $\overline{w}: S^1 \to S^1$.  

With this as background we have:

\begin{defin} \label{defin:winding} The winding number 
\[ 
\omega([b]) \in \Zzz
\]
is the degree of the map $\overline{w}$.
\end{defin}

Using this definition the following Proposition is elementary. 

\begin{prop}  \label{prop:windtrivial} Let $b: I \to F_2\Rrr^2$ be a pure 2-braid and  let $d: \Rrr^2 \times \Rrr^2 \to [0, \infty)$ be either the euclidean or hyperbolic metric on $\Rrr^2$.  Suppose, for each $t \in I$ and each $i = 1, 2$, $d(x_i, b(t)(x_i))  < d(x_1, x_2)/2$.  

Then $\omega([b]) = 0$.
\end{prop}

\begin{proof}  Let $L$ be the geodesic in $\Rrr^2$ that orthogonally bisects the geodesic segment connecting the points $x_1$ and $x_2$.  Then each point $\ell \in L$ is equidistant from $x_1$ and $x_2$,  and $d(\ell, x_i) \geq d(x_1, x_2)/2$ for $i = 1, 2$.  So the hypothesis implies that, for all $t$, the two points $b(t)(x_1)$ and $b(t)(x_2)$ lie on opposite sides of $L$.  See Figure \ref{fig:Gold1}. In particular, the geodesic ray from $b(t)(x_1)$ to $b(t)(x_2)$ never achieves the angle (when $d$ is euclidean) or the limit point at $\infty$ in $S^1$ (when $d$ is hyperbolic) of any infinite ray in the line $L$.  That is, the corresponding map $\overline{w}: S^1 \to S^1$ is not surjective, so $\overline{w}$ has trivial degree, completing the proof.
\end{proof}

    
 \begin{figure}[H]
  \centering
   \begin{tikzpicture}
   \usetikzlibrary {arrows.meta} 
\draw (0,-2.5) -- (0,2.5);
\draw[dashed] (0,-2.5) -- (0,-3.5);
\draw[dashed] (0,2.5) -- (0,3.5);
\draw[loosely dashed]  (0,0) circle (3.5); 
\draw  (-1.2,0) circle (1.2); 
\draw  (1.2,0) circle (1.2); 
 \draw [arrows = {-Latex[width=8pt, length=8pt]}] (-.8,0.1) arc [start angle=-80, end angle=-40, radius=6]; 
  \filldraw [gray] (-1.4,0) circle [radius=1.5pt]; 
 \filldraw [gray] (1.4,0) circle [radius=1.5pt];  
 \filldraw [gray] (-.8,0.1) circle [radius=1.5pt]; 
 \filldraw [gray] (.8,0.6) circle [radius=1.5pt];  
  \node at (-0.8,0.4) {$b(t)(x_1)$};
 \node at (1.6,0.6) {$b(t)(x_2)$};
   \node at (-1.4,-0.2) {$x_1$};
  \node at (1.6,-0.2) {$x_2$};
  \node at (2.0,1.7) {$r(t)$};
   \node at (3.2,2.2) {$w(t)$};
 \node at (-0.2,-1.5) {$L$};
\node at (0,-3.8) {};
\end{tikzpicture}
   \caption{The map $w: I \to S^1$ (hyperbolic case)}
    \label{fig:Gold1}
\end{figure}

 \section{The breadth of braids and braid classes}
 
For $M$ a closed surface not the sphere or projective plane, fix a constant curvature  metric on $M$, euclidean if $\chi(M) = 0$, hyperbolic if $\chi(M) <0$.  Lift that metric to the universal cover $\tM$ of $M$, giving $\tM$ the structure of either the euclidean or hyperbolic plane.  Denote by $d: \tM \times \tM \to [0, \infty)$ this standard metric. 
With no loss, assume the points $X = \{x_1, x_2, ..., x_n\}$ lie in the interior of an evenly covered closed disk $D \subset M$ of radius $\rho >0 $. 
For each $x_i \neq x_j \in X$ there is a natural isomorphism $\pi_1(M, x_i) \cong \pi_1(M, x_j)$ obtained by conjugating a loop at $x_i$ by a path in $D$ between $x_i$ and $x_j$.  When the choice of base point $x_i \in D$ is unimportant we will use the generic notation $\pi_1(M, x_\bullet)$.\

For $p: \tM \to M$ the covering projection fix a component $\tD$  of $p^{-1}(D)$ and, for $1 \leq i \leq n$ let $\tx_i$ denote the lift of $x_i$ that lies in $\tD$ and let $\tX = \{\tx_1, \tx_2, ..., \tx_n\}$.  For each $\aaa \in \pi_1(M,x_\bullet )$ let $\tD(\aaa)$ denote the image of $\tD$ under the covering translation of $\tM$ determined by $\aaa$.  Similarly, let $\tx_i(\aaa)$ denote the lift of $x_i$ that lies in $\tD(\aaa)$.  In particular, for $e \in \pi_1(M,x_\bullet)$ the identity, $\tD = \tD(e)$ and $\tx_i(e) = \tx_i, 1 \leq i \leq n$.

A {\em pure braid of $n$ strands}  in $M$  is a map $b: I \to F_nM$ so that $b(0) = b(1) = x^0$.  Denote by $[b]$ the class of $b$ in $\pi_1(F_nM, x^0)$.
For each $1 \leq i \leq n$ the path $b(t)(x_i), t \in I$ is a loop in $M$ based at $x_i$.  Denote by $\tb_i: I \to \tM$ the path that is the lift of the loop $b(t)(x_i)$ for which $\tb_i(0) \in \tD$.  Then $\tb_i(1)$ lies in a lift $\tD(\aaa_i)$ of $D$ in $\tM$, for some $\aaa_i \in \pi_1(M, x_i)$; Goldberg defines 
\[g_*([b]) = (\aaa_1, ..., \aaa_n) \in \Pi_{i=1}^n \pi_1(M, x_i).
\]
In particular, if the braid class $[b] \in K = ker(g_*)$ then each $\aaa_i = e$ so each path $\tb_i$ is a loop in $\tM$.  That is, $\tb = (\tb_1, ..., \tb_n) \in F_n\tM$ is itself a braid on the points $\tX \subset \tD \subset \tM$.  By the covering homotopy property, the class $[\tb] \in P_n(\tM) = \pi_1(F_n\tM, \tx^0)$ is well defined.  (Here $\tx^0$ denotes the inclusion of $\tX$ into $\tD \subset \tM$.)

\begin{defin} For any braid $b$ such that $[b] \in K \subset P_n(M)$ let $\tb:I \to F_n\tM$ be the braid in $\tM$ just defined.  Then the {\em breadth} $\nu(b) \in [0, \infty)$ of $b$ is given by
\[
\nu(b) = \max_{i\in \{1,...,n\}} \max_{t \in I}d(\tx_i, \tb_i(t)) 
\]
\end{defin}

Here are two obvious properties:

\begin{lemma}  For any braid $b$ such that $[b] \in K \subset P_n(M)$, if $\nu(b) = 0$ then $b$ is the constant braid in $M$.
\end{lemma}
\begin{proof}  If $\nu(b) = 0$ then for each $i$ and every $t \in I$, $\tb(t)(\tx_i) = \tx_i$.  This implies that for every $t$, $\tb(t)$ is just the inclusion, so $\tb$ is the constant braid in $\tM$.  Then its projection $b$ is the constant braid in $M$, as claimed.  
\end{proof}

\begin{lemma}  \label{lemma:braidproduct1} Suppose $b$ and $b'$ are two braids so that $[b], [b'] \in K$ and $b \cdot b'$ is their braid product.  Then 
\[ \nu(b \cdot b') = \max\{\nu(b), \nu(b')\} \]
\end{lemma}

\begin{proof}  Let $\tb$ and $\tb'$ be the lifts of $b$ and $b'$ respectively.  Then the braid product $\tb \cdot \tb'$ is the lift of $b \cdot b'$.  Moreover for each $i$, the function $d(\tx_i, (\tb \cdot \tb')(t)(\tx_i))$ takes on the union of the values taken by $d(\tx_i, \tb(t)(\tx_i))$ and $d(\tx_i, \tb'(t)(\tx_i))$.  So the maximum over all values is the same, as claimed. 
\end{proof}

Similarly

\begin{defin} For any braid class $[b] \in K \subset P_n(M)$ 
define the {\em breadth} $\nu([b]) \in [0, \infty)$ by
\[
\nu([b]) = \inf_{b \in [b]} \nu(b).
\]
\end{defin}

Again this leads to two obvious lemmas:

\begin{lemma}  For any braid class $[b] \in K \subset P_n(M)$, if $\nu([b]) = 0$ then $[b]$ is the identity in $P_n(M)$.
\end{lemma}

\begin{proof}  Let $\eee > 0$ be so small that an $\eee$-neighborhood of $X \subset M$ (that is, a collection of disjoint disks centered on the $x_i$, each of radius $\eee$) deformation retracts to $X$. Then a lift $\tilde{\eta}$ of this neighborhood to $\tD \subset \tM$ consists of a disjoint collection of $\eee$-disks centered on the $\tx_i$. If $\nu([b]) = 0$ then there must be a braid $b \in [b]$ such that $\nu(b) < \eee$.  That is, for each $i \in \{1, ..., n\}$ and $t \in I$, $d(\tx_i, \tb_i(t)) < \eee$. Then the deformation retraction of $\tilde{\eta}$ to $\tX$ homotopes $\tb$ to the constant braid in $\tM$ and its projection homotopes $b$ to the constant braid in $M$, as claimed.
\end{proof}

\begin{lemma}  \label{lemma:braidproduct2} Suppose $b$ and $b'$ are two braids so that $[b], [b'] \in K$.  Then 
\[ \nu([b \cdot b']) \leq \max\{\nu([b]), \nu([b'])\} \]
\end{lemma}

\begin{proof}  For any $\eee > 0$ there are, by definition, representative braids $b \in [b]$ and $b' \in [b']$ so that $\nu(b) < \nu([b])) + \eee$ and $\nu(b') < \nu([b'])) + \eee$.  Then by Lemma \ref{lemma:braidproduct1}, 
\[
\nu([b \cdot b']) \leq \nu(b \cdot b') = \max\{\nu(b), \nu(b')\} \leq \max\{\nu([b]), \nu([b'])\} + \epsilon
\]
Since this is true for all $\eee > 0$ we conclude $\nu([b \cdot b'])  \leq \max\{\nu([b]), \nu([b'])\}$, as required.  
\end{proof}

\begin{cor}   \label{cor:boundedbreadth} Suppose the kernel $K$ is finitely generated.  Then there is a $C >0$ so that for any $[b] \in K$,
$\nu([b]) < C$.  
\end{cor}

\begin{proof}   Suppose $[b_1], ..., [b_m]$ are a finite set of generators for $K$.  Then any braid class $[b] \in K$ can be written as the braid product of copies of the $[b_j], 1 \leq j \leq m$.  It follows then from \ref{lemma:braidproduct2} that $\nu([b]) \leq \max_{1 \leq j \leq m} \nu([b_j])$, as claimed.  
\end{proof}

\section{Braid breadth is unlimited}

Following Corollary \ref{cor:boundedbreadth}, to prove Theorem \ref{thm:main} it suffices to prove the following proposition

\begin{prop} \label{prop:mainprop}   For any $C > 0$ there is a braid class $[b] \in P_n(M)$ so that $\nu([b]) \geq C$.
\end{prop}

\begin{proof}  Given $C$, we will find 
 a braid $b$ so that $\nu(b) \geq C$, then show that any homotopic braid has the same property.  
 We may as well take $C$ to be very large compared to the radius $\rho$ of $D \subset M$.  

\begin{lemma} \label{lemma:primalpha}  There is a primitive $\gamma \in \pi_1(M, x_\bullet)$ so that in the euclidean or hyperbolic plane $\tM$ the lift $\tD(\gamma)$ of $D$ has distance from $\tD$ greater than $2C$.  
\end{lemma}

\begin{proof}  Consider the family $\{r_{\theta}, \theta \in S^1\}$ of geodesic rays originating at $\tx_1 \subset \tM$, parameterized so that $r_\theta$ leaves the point $\tx_1$ at the angle $\theta$.  
\medbreak

{\em Claim:} The set of angles $\theta$ for which $r_\theta$ passes through a lift of $x_1$ in $\tM$ is dense in $S^1$.

{\em Proof of claim:} If not, there would be a subinterval $(\theta_-, \theta_+) \subset S^1$ such that, for every $ \theta \in (\theta_-, \theta_+)$, the ray $r_\theta$ does not intersect any lift of $x_1$ other than at the ray's origin $\tx_1$.  But this would imply that some fundamental domain of $M$ in $\tM$, perhaps at a great distance from $\tx_1$, would contain no lift of $x_1$.See Figure \ref{fig:Gold2}.  This is absurd and so proves the claim.

    
     \begin{figure}
  \centering
   \begin{tikzpicture}
\usetikzlibrary{arrows,decorations.markings}
\usetikzlibrary {shapes.geometric}
\draw [decoration={markings,mark=at position 1 with
    {\arrow[scale=3,>=stealth]{>}}},postaction={decorate}] (0,0) -- (0,2.5);
\draw [decoration={markings,mark=at position 1 with
    {\arrow[scale=3,>=stealth]{>}}},postaction={decorate}] (0,0) -- (0.9,2.3);
\draw[dashed] (0,2.5) -- (0,3.5);
\draw[dashed] (0.8,2.) -- (1.35,3.17);
 \filldraw [lightgray] (0.7,3.0) ellipse [x radius=8pt, y radius=5pt]; 
\draw[loosely dashed]  (0,0) circle (3.5); 
\node[regular polygon, regular polygon sides=8, draw,
     inner sep=0.7cm] at (0,0) {};  
  \filldraw [gray] (0,0) circle [radius=1.5pt]; 
   \node at (0,-0.3) {$\tx_1$};
 \filldraw [gray] (0,-2.5) circle [radius=1.5pt];  
  \node at (0,-3) {$\tx_1(\alpha)$};
\filldraw [gray] (1.5,-1.5) circle [radius=1.5pt];  
  \node at (1.5,-1.8) {$\tx_1(\beta)$};
  \node at (-0.4,2) {$r(\theta_+)$};
  \node at (1.5,2) {$r(\theta_-)$};
  \node at (0,-3.8) {};
\end{tikzpicture}
   \caption{$\tx_1(\alpha)$, $\tx_1(\beta)$ are sample lifts of $x_1$ for $\alpha, \beta \in \pi_1(M, \bullet)$}
    \label{fig:Gold2}
\end{figure}

%
%

The set of lifts of $x_1$ that lie in a disk of radius $3C$ centered at $\tx_1$ is clearly finite.  So, following the claim,  it is possible to choose a ray $r_\theta$ that first encounters any lift of $x_1$ at a distance greater than $3C$ from the ray's origin $\tx_1$.  The lift of $\tx_1$ that is first encountered by $r_\theta$ determines a primitive $\gamma \in \pi_1(M, x_\bullet)$ so that $d(\tx_1, \tx_1(\gamma)) > 3C$. See Figure \ref{fig:Gold3}. 
 Then the corresponding disk $\tD(\gamma)$ is at distance at least $2C$ from $\tD$, proving Lemma \ref{lemma:primalpha}.  \end{proof}

%
    
  \begin{figure}
  \centering
   \begin{tikzpicture}
   \usetikzlibrary {shapes.geometric}
\draw[loosely dashed]  (0,0) circle (3.5); 
\draw  (0,0) circle (1.5); 
 \draw (0.75,0) -- (1.5,0); 
 \draw[->] (0,0) -- (.75,0); 
   \node at (0.75,0.3) {$3C$};
  \filldraw [gray] (0,0) circle [radius=1.5pt]; 
   \node at (0,0.3) {$\tx_1$};
 \filldraw [gray] (-0.2,-2.5) circle [radius=1.5pt];  
  \node at (-0.2,-2.8) {$\tx_1(\gamma)$};
     inner sep=0.3cm] at (0,0) {}; 
  \node at (0.3,-2) {$r(\theta)$};
   \draw[->] (0,0) -- (-0.1,-1.25) ; 
   \draw (-0.1,-1.25) -- (-0.2,-2.5) ;
\end{tikzpicture}
   \caption{Finding primitive $\gamma \in \pi_1(M, x_1)$}
    \label{fig:Gold3}
\end{figure}

Let $\tL$ of length $|L| > 2C$ be a geodesic segment in $\tM$ from $\tx_1 \in \tD$ to $\tx_2(\gamma) \in \tD(\gamma)$.  Let $L$ denote the (embedded) image of $\tL$ in $M$.  By generic choice of $X$ we can assume that $L$ is disjoint from $X$ other than at the endpoints $\{x_1, x_2\}$ of $L$.  Since $\gamma$ is primitive, $\tL$ is disjoint from every other lift of $L$ in $\tM$.

Choose $\ddd > 0$ to be less then the minimum distance between $L$ and $X - \{x_1, x_2\}$ and also less than the distance in $\tM$ between $\tL$ and any other lift of $L$.  Then a  $\ddd/2$-neighborhood $E \subset M$ of $L$ in $M$ is a topological disk 
that intersects $X$ only in $\{x_1, x_2\}$. 

Let $\tE \subset \tM$ be the lift of $E$ that is a $\ddd/2$-neighborhood of $\tL$, and let $\tb_E: I \to F_2\tE$ be a pure braid on $\{\tx_1, \tx_2(\gamma)\}$ in $\tE \subset \tM$ that has winding number 1.  

 \begin{lemma}  \label{lemma:bigbreadth} For some $t \in [0, 1]$, either $d(\tx_1, \tb_E(t)(\tx_1))$ or $d(\tx_2, \tb_E(t)(\tx_2)) \geq C$
    \end{lemma} 
    \begin{proof}  Since $\omega([\tb_E]) \neq 0$ it follows from Proposition \ref{prop:windtrivial} that for some $t$, either $d(\tx_1, \tb_E(t)(\tx_1))$ or $d(\tx_2(\gamma), \tb_E(t)(\tx_2(\gamma)) \geq C$.  But since covering transformations are isometries, $d(\tx_2(\gamma), \tb_E(t)(\tx_2(\gamma)) = d(\tx_2, \tb_E(t)(\tx_2))$, completing the proof.
    \end{proof}

Let  $b_E: I \to F_2E \subset F_2M$ be the pure braid on $\{x_1, x_2\}$ in $E$ that is the  projection of $\tb_E$.    That is, for all $t \in I$
 \begin{align*}
 b_E(t)(x_1) &= p\tb_E(t)(\tx_1)  \\
 b_E(t)(x_2) &= p\tb_E(t)(\tx_2(\gamma))
  \end{align*}
  Let $b_L: I \to F_nM$ be the braid defined by
 \begin{align*}
 b_L(t)(x_i) &=  b_E(t)(x_i), &i &= 1, 2 \\
  b_L(t)(x_i)  &=  x_i, &i &= 3, ..., n
  \end{align*}
  
  It then follows from Lemma \ref{lemma:bigbreadth} that the breadth $\nu(b_L) \geq C$.  \begin{lemma} Suppose $b:I \to F_nM$ is any pure braid homotopic to $b_L$.  Then $\nu(b) \geq C$.
 \end{lemma}
 \begin{proof}  Since $b$ is homotopic to $b_L$ the 2-braid in $M$ that is the restriction of $b$ to $\{x_1, x_2\}$ is homotopic to $b_E$. Then its lift to $\tM$ is also homotopic to $\tb_E$ and so has winding number 1.  Now apply the same argument to $b$ as was just used for $b_L$ and conclude that also $\nu(b) \geq C$, as claimed.  
 \end{proof}
 
The last lemma implies that $\nu([b_L]) \geq C$, completing the proof of Proposition \ref{prop:mainprop}.
\end{proof}

\end{document}